\documentclass[a4paper, 12pt]{amsart}
\usepackage{amsmath}
\usepackage{hyperref}
\usepackage{graphicx}
 \usepackage{amsfonts}
\usepackage{amssymb,comment}

\newtheorem{theorem}{Theorem}[section]

\newtheorem{proposition}[theorem]{Proposition}
\newtheorem{corollary}[theorem]{Corollary}
\theoremstyle{definition}
\newtheorem{definition}[theorem]{Definition}
\newtheorem{remark}[theorem]{Remark}
\newtheorem{example}[theorem]{Example}

\newcommand{\innerproduct}[2]{\left\langle #1 , #2\right\rangle}
\newcommand{\romd}{\,\mathrm{d}}
\newcommand{\norm}[1]{\left\|#1\right\|}
\newcommand{\e}{e_F} 

\makeatletter
\@namedef{subjclassname@2020}{%
  \textup{2020} Mathematics Subject Classification}
\makeatother

\begin{document}
\date{2020-4-13}
\title[Relative discrete spectrum]{Relative discrete spectrum of 
W*-dynamical systems}

\author[R. Duvenhage]{Rocco Duvenhage}
\address{Department of Physics\\ University of Pretoria\\
Pretoria 0002, South Africa}
\email{rocco.duvenhage@up.ac.za}

\author[M. King]{Malcolm King}
\address{Department of Mathematics and Applied Mathematics\\ University of Pretoria\\
Pretoria 0002, South Africa}
\email{malcolmbruceking@gmail.com}

\begin{abstract}
A definition of relative discrete spectrum of noncommutative W*-dynamical systems is given
in terms of the basic construction of von Neumann algebras, motivated from three perspectives:
Firstly, as a complementary concept to relative weak mixing of W*-dynamical systems. Secondly,
by comparison with the classical (i.e. commutative) case. And, thirdly, by noncommutative examples.

\end{abstract}
\subjclass[2020]{Primary 46L55}

 \keywords{W*-dynamical systems; relative discrete spectrum; relative weak 
mixing; relatively 
independent joinings.}
\maketitle

\section{Introduction}
\label{afdInl}
In his study of ergodic actions of locally compact groups, Zimmer
\cite{Z1,Z2}
introduced relative discrete spectrum and proved what was to become known 
as the Furstenberg-Zimmer Structure Theorem. 
Proving the same structure theorem independently, Furstenberg \cite{F77} gave an 
ergodic theoretic proof of Szemeredi's Theorem.

In the noncommutative setting of W*-dynamical systems, Austin, Eisner and Tao \cite{AET}
proved a partial analogue of the Furstenberg-Zimmer Structure Theorem, providing conditions under
which a certain case of relative weak mixing holds.
In their approach, which builds on the work by Popa \cite{P}, the basic construction of von Neumann 
algebras is an essential tool, although they do not define relative weak mixing in terms
of the basic construction, and do not define relative discrete spectrum at all. Their use of the basic 
construction forms the basis for our approach to relative discrete spectrum, where we employ the basic 
construction for the von Neumann algebra of a W*-dynamical system and the subalgebra relative to which we 
want to define discrete spectrum of the W*-dynamical system. Of particular importance is 
\cite{AET}'s characterization of systems which are not relatively weakly mixing in 
terms of the existence of a non-trivial submodule, invariant under the dynamics and 
finite with respect to the trace on the basic construction. In the noncommutative case these kinds 
of submodules play an analogous role to the finite rank submodules  
which appear in the classical case.

The paper has two main parts. The first, consisting of Sections 2 and 3, treats our noncommutative 
definition of relative discrete 
spectrum. The definition is given in terms of the basic construction, and is motivated by the need 
to make relative discrete spectrum
complementary to relative weak mixing as in the classical case. Some tools and ideas provided by the
theory of joinings of W*-dynamical systems are used in the process. Our 
definition is then shown to not only be 
a noncommutative 
generalization of 
classical relative discrete spectrum, but also to generalize the noncommutative 
version of (absolute) discrete spectrum.

In the second part, consisting of Sections \ref{afdEU} and 
\ref{sec_finite_ext}, we discuss two noncommutative 
examples of relative discrete spectrum.
The first example (Section \ref{afdEU}) is a skew product of a commutative system with a 
noncommutative one. The second (Section \ref{sec_finite_ext}) is a purely 
noncommutative example on the von Neumann tensor product of two 
noncommutative systems, where the second 
system is finite dimensional.

We end the paper with a brief discussion of some open problems (Section \ref{sec_questions}).

Throughout this paper we will be working only with traces on von Neumann algebras, not general states or 
weights.
Note that we use the convention where inner products are linear in the right
and conjugate linear in the left.

\section{Relatively Independent Joinings and Relative Weak Mixing}
As the first step towards the concept of relative discrete spectrum, we study how relatively 
independent joinings (see \cite{D3, BCM}) can be expressed in terms of the basic construction.
Combining this with theory from \cite{DK} regarding relative weak mixing, places us in a position to
proceed to relative discrete spectrum in the next section.

In the remainder of this paper W*-dynamical systems are referred to 
as ``systems'' and we define them as follows:

\begin{definition}
\label{stelsel}A \emph{system} $\mathbf{A }=\left(  A ,\mu,\alpha\right)  $
consists of a faithful normal trace $\mu$ on a (necessarily finite) von
Neumann algebra $A $, and a $\ast$-automorphism $\alpha$ of $A $, such that
$\mu\circ\alpha=\mu$.
\end{definition}

In the sequel, for
$\mathbf{A }$ we assume without loss that $A $ is a von Neumann algebra on the
Hilbert space $H  $, with $\mu$ given by a cyclic and separating vector
$\Omega\in H  $, i.e.
\[
\mu(a)=\innerproduct{\Omega}{a\Omega}
\]
for all $a\in A $.

The dynamics $\alpha$ of a system $\mathbf{A }$ can be represented by a unitary
operator $U$ on $H  $ defined by extending
\[
Ua\Omega:=\alpha(a)\Omega.
\]
It satisfies
\[
Ua U^{\ast}=\alpha(a)
\]
for all $a\in A $.

Along with $\mathbf{A}$ above, we also use the notation
\[\mathbf{B}=(B, \nu, \beta)\qquad\text{and }\qquad \mathbf{F}=(F,\lambda,\varphi)\]
to denote systems.

\begin{definition}
We call $\mathbf{F}$ a \emph{subsystem} of $\mathbf{A }$ if $F$ is a von
Neumann subalgebra of $A $ (containing the unit of $A $) such that $\mu
|_{F}=\lambda$ and $\alpha|_{F}=\varphi$.
\end{definition}

Throughout the rest of the paper, $\mathbf{F}$ will be a subsystem of
$\mathbf{A }$. Set
\[
H_{F}:=\overline{F\Omega}.
\]

Next we review elements of the basic construction and relatively independent joinings.
Let $\e$ denote the projection of $H$ onto $H_F$.
    We consider the basic construction, $\innerproduct{A }{\e}$, the 
smallest von
    Neumann algebra (in $\mathcal{B}(H))$ containing $A $ and $\e.$ 
    See  \cite{Sk}, \cite{C} and \cite{J}.
    

Since $\mu$ is a trace,
we obtain from it a faithful semifinite normal tracial weight 
$\bar{\mu}:\innerproduct{A }{\e}^{+}\rightarrow\lbrack0,\infty]$. 
It is also defined and tracial on the strongly dense $\ast$-subalgebra 
$A  \e A :=\operatorname{span}\{a\e b:a,b\in A \}$ of 
$\innerproduct{A }{\e}$ via the equation
\[
\bar{\mu}(a\e b)=\mu(ab).
\]
For more on the basic construction and the trace $\bar{\mu}$, see
\cite[Chapter 4]{SS}.

We can extend the dynamics of $\alpha$ to $\innerproduct{A }{\e}$ 
by
\[
\bar{\alpha}(a)=U a U^{\ast}
\]
for $a\in\innerproduct{A }{\e}$. 
Then from \cite[Section 3]{DK},
\[\bar\mu\circ \bar\alpha=\bar\mu.\]

Furthermore, we have a unitary operator
\[\bar{U}:\bar {H}\rightarrow\bar{H}\] 
representing $\bar{\alpha}$ on the Hilbert space
$\bar{H}$ obtained from the GNS construction for $(\innerproduct{A }{\e},\bar{\mu})$. 
 Denoting the
quotient map of this construction as 
\begin{equation}\label{eq_quotient_map_semicyclic}
\gamma_{\bar\mu}:\mathcal{N}_{\bar\mu}\rightarrow\bar{H},
\end{equation}
where 
\begin{equation}\label{eq_quotient_ideal_semicyclic}
\mathcal{N}_{\bar\mu}:=\{a\in\innerproduct{A }{\e}:\bar{\mu}(a^{\ast}a)<\infty\},
\end{equation}
we define $\bar{U}:\bar{H}\rightarrow\bar{H}$ via
\[
\bar{U}\gamma_{\bar\mu}(a)=\gamma_{\bar\mu}(\alpha(a)).
\]

We now turn to the relatively independent joining 
and its relation to the basic 
construction.
The modular conjugation associated to the trace $\mu$, will be denoted by $J$.
We let
\[
j:\mathcal{B}(H  )\rightarrow \mathcal{B} (H  ):a\mapsto Ja^{\ast}J,
\]
where $\mathcal{B}(H )$ is the von Neumann algebra of all bounded linear operators on $H$.
Carry the trace and dynamics of the system $\mathbf{A }$ over
to $A ^{\prime}$ in a natural way using $j$, by defining a trace $\mu^{\prime}$
and $\ast$-automorphism $\alpha^{\prime}$ on $A ^{\prime}$ by
\[
\mu^{\prime}(b):=\mu\circ j(b)=\innerproduct{\Omega}{b\Omega}
\]
and
\[
\alpha^{\prime}(b):=j\circ\alpha\circ j(b)=UbU^{\ast}%
\]
for all $b\in A ^{\prime}$ (where we made use of $UJ=JU$). This defines the system
\[
\mathbf{A }^{\prime}:=(A ^{\prime},\mu^{\prime},\alpha^{\prime}).
\]

Set
\[
\tilde{F}:=j(F),
\]
\[
\tilde{\lambda}:=\mu^{\prime}|_{\tilde{F}},%
\]
and
\[
\tilde{\varphi}:=\alpha^{\prime}|_{\tilde{F}}.
\]

Let 
\[
D:A \rightarrow F
\]
be the unique conditional expectation
such that $\lambda\circ D=\mu$. Then
\[
{D^{\prime}}:=j\circ D\circ j:A ^{\prime}\rightarrow\tilde{F}%
\]
is the unique conditional expectation such that $\tilde{\lambda}%
\circ{D^{\prime}}=\mu^{\prime}$.
For later use we note that, since $j(f)\Omega=Jf^*\Omega=f\Omega$ for all 
$f\in F$, we have
\begin{equation}\label{eq_Dprime}
D'(b)\Omega=D(j(b))\Omega 
\end{equation}
for all $b\in A'$.
Define the unital $\ast$-homomorphism
\[
\delta:F\odot\tilde{F}\rightarrow B(H),
\]
on the algebraic tensor product $F\odot\tilde{F}$ as the linear extension of $F\times\tilde{F}\rightarrow 
B(H):(a,b)\mapsto
ab$. Define the \emph{diagonal} state
\[
\Delta_{\lambda}:F\odot\tilde{F}\rightarrow\mathbb{C}%
\]
of $\lambda$ by
\[
\Delta_{\lambda}(c):=\innerproduct{\Omega}{\delta(c)\Omega}
\]
for all $c\in F\odot\tilde{F}$. The \emph{relatively independent joining of }$\mathbf{A }$\emph{\ and }
$\mathbf{A }^{\prime}$\emph{\ over }$\mathbf{F}$ is the state $\mu\odot
_{\lambda}\mu^{\prime}$ on $A \odot A ^{\prime}$ given by
\begin{equation}
\mu\odot_{\lambda}\mu^{\prime}:=\Delta_{\lambda}\circ D\odot{D^{\prime}}. \label{mumu}%
\end{equation}
Subsequently we denote this joining
by
\[
\omega:=\mu\odot_{\lambda}\mu^{\prime}
\]
and also write
\[
\mathbf{A }\odot_{\mathbf{F}}\mathbf{A }^{\prime} :=(A \odot A 
^{\prime}, \omega,\alpha\odot\alpha^{\prime}).
\]
The cyclic representation of $(A \odot A ^{\prime}, \omega)$ obtained by the
GNS construction will be denoted by $(H_{\omega},\pi_{\omega},\Omega_{\omega
})$. 
Let
\[
\gamma_{\omega}:A \odot A ^{\prime}\rightarrow H_{\omega}:t\mapsto\pi_{\omega
}(t)\Omega_{\omega}.
\]
By $W$ we denote the unitary representation of
\[
\tau:=\alpha\odot\alpha^{\prime}%
\]
on $H_{\omega}$ defined as the extension of
\[
W\gamma_{\omega}(t):=\gamma_{\omega}(\tau(t))
\]
for all $t\in A \odot A ^{\prime}$.

We also set  
\begin{equation}\label{eq_Hlambda}
H_{\lambda}:=\overline{\gamma_{\omega}(F\otimes1)}.
\end{equation}

Next we turn our attention to expressing the GNS representation of $\omega$ in terms of $\bar{H}$, 
which is convenient for our subsequent work. The key point is to construct a natural unitary equivalence 
$R:H_\omega\rightarrow \bar{H}$ between $W$ and $\bar{U}$. 
In the classical case, such a result appears in 
\cite[pp. 63--64]{petersonErgodicNotes}.

\begin{proposition}\label{prop_R}
We have a uniquely determined well-defined unitary operator
\[R:H_\omega \rightarrow \bar{H}\]
satisfying 
$R\gamma_\omega(a\otimes j(b))=\gamma_{\bar\mu}(a\e b)$ for all $a,b\in A.$

Furthermore, 
\[\bar{U}=RWR^*.\]
\end{proposition}
\begin{proof}
Since $j$ is linear, we may define $R_0:A \odot A'\rightarrow \innerproduct{A }{\e}$ via the prescription
\[R_0(a\otimes b):=a\e j(b)\]
for $a\in A $ and $b\in A'$.
From the universal property of $A \odot A ',$ $R_0$ is well-defined and 
linear.
Note that $R_0(A \otimes A ')\subset \mathcal{N}_{\bar\mu}$ with 
$\mathcal{N}_{\bar\mu}=\{x\in\innerproduct{A }{\e}:\bar{\mu}(x^{\ast}x)<\infty\}$ 
as in \eqref{eq_quotient_ideal_semicyclic}.
Hence, we can consider 
\[R:\gamma_{\bar\mu}(A\odot A')\rightarrow \bar{H}:\gamma_\omega(t)\mapsto 
\gamma_{\bar\mu}(R_0(t)).\]

 We need to show that $R$ is well-defined and uniquely extends to a unitary operator $H_\omega\rightarrow 
\bar{H}$. For clarity, below, we distinguish the inner products of $H_\omega$ and
$\bar{H}$ by subscripts $\omega$ and $\bar\mu$. Note that for $a,c\in A$ and $b,d\in A'$,
 \begin{align}
     \innerproduct{\gamma_{\bar\mu}(R_0(a\otimes b))}{\gamma_{\bar\mu}(R_0(c\otimes 
d))}_{\bar\mu}
     =&\innerproduct{\gamma_{\bar\mu}(a\e j(b))}{\gamma_{\bar\mu}(c\e j(d))}_{\bar\mu}\nonumber\\
	=&\bar\mu(j(b^{\ast})\e a^{\ast} c\e j(d))\nonumber\\
	=&\bar\mu(\e a^{\ast}c\e j(d)j(b^{\ast})\e)\nonumber\\
    =&\bar\mu(D(a^{\ast}c)\e D(j(b^{\ast}d))\nonumber\\
    =&\mu(D(a^{\ast}c)D(j(b^{\ast}d)))\nonumber\\
	=&\innerproduct{\Omega}{D(a^{\ast}c)D'(b^{\ast}d)\Omega}\nonumber\\
	=&\innerproduct{\Omega}{\delta\circ(D\odot D')((a^{\ast}c)\otimes (b^{\ast}d))\Omega}\nonumber\\
	=&\omega((a^{\ast}c)\otimes (b^{\ast}d))=\omega((a\otimes b)^{\ast}(c\otimes d))\nonumber\\
	=&\innerproduct{\gamma_\omega(a\otimes b)}{\gamma_\omega(c\otimes d)}_\omega,\nonumber
 \end{align}
where we have used \eqref{eq_Dprime}.
 So it follows that for all $s,t\in A \odot_F A ',$
 \begin{equation}\label{eq_R_0}
\innerproduct{\gamma_{\bar\mu}(R_0(s))}{\gamma_{\bar\mu}(R_0(t))}_{\bar\mu}=\innerproduct{
\gamma_\omega(s) } {
\gamma_\omega(t) } _\omega.
 \end{equation}
 Thus, $R$ is well-defined (as $\gamma_\omega(t)=0$ implies $\gamma_{\bar\mu}(R_0(t))=0$) and can be extended 
to an isometric linear operator, still denoted by $R,$ from
 $H_\omega$ to $\bar{H}.$ 
From \cite[Lemma 4.3.10]{SS}, $\gamma_{\bar\mu}(A \e A )$ is 
dense in 
$\bar{H}.$
 It follows that $R\gamma_\omega(A \odot A ')=\gamma_{\bar\mu}(R_0(A \odot 
A '))=\gamma_{\bar\mu}(A \e A )$ is
 dense in $\bar{H}.$ Hence, $RH_\omega=\bar{H}$ and therefore $R$ is a unitary operator.

 For $a,b\in A ,$
\begin{align}
	RWR^{\ast}(\gamma_{\bar\mu}(a\e b))=&RW\gamma_\omega(a\otimes 
j(b))=R\gamma_\omega(\alpha(a)\otimes 
j(\alpha(b)))\nonumber\\
	=&\gamma_{\bar\mu}(\alpha(a)\e\alpha(b))=\gamma_{\bar\mu}(\bar\alpha(a\e b))\nonumber\\
	=&\bar{U}(\gamma_{\bar\mu}(a\e b)),\nonumber
\end{align}
which implies that
$	\bar{U}=RWR^{\ast}$.
\end{proof}

Note that we can express the relatively independent joining in terms of $\bar\mu$ using $R$:
For all $a\in A$ and $b\in A'$, 
\begin{align}\omega(a\otimes b)=&\innerproduct{R\gamma_\omega(1)}{R\gamma_\omega(a\otimes 
b)}_{\bar\mu}=\innerproduct{\gamma_{\bar\mu}(\e)}{\gamma_{\bar\mu}(a\e j(b))}_{\bar\mu}
\nonumber\\
=&\bar\mu(\e a\e 
j(b))=\bar\mu(D(a)\e D(j(b)).\nonumber
\end{align}

If $H^W_\omega$ denotes the vector space of all fixed points of $W,$ then
\[\bar{H}^{\bar{U}}:=RH^W_\omega,\]
must be the fixed points of $\bar{U}$.
We also have a copy of $H_\lambda$ in $\bar{H}$:
\begin{equation}\label{eq_eF}
\begin{aligned}
    \bar{H}_\lambda:=&RH_\lambda
    =R\overline{\gamma_\omega(1\otimes \tilde{F})}&&\text{from }\eqref{eq_Hlambda}\\
=&\overline{R\gamma_\omega(1\otimes \tilde{F})}\\
=&\overline{\gamma_{\bar\mu} [R_0(1\otimes \tilde{F})]}\\
=&\overline{\gamma_{\bar\mu}(\e F)}.
\end{aligned}
\end{equation}

Having obtained our unitary equivalence $R$ in Proposition \ref{prop_R}, we can rephrase relative ergodicity
(\cite[Definition 4.1]{DK})
from a ``basic construction'' point of view:
\begin{definition}\label{def_mod_rel_ergodicity}
 We say that $\mathbf{A \odot_F A '}$ is ergodic relative to a subsystem
 $\mathbf{F}$ of $\mathbf{A}$, if $\bar{H}^{\bar{U}}\subset \bar{H}_\lambda.$
\end{definition}

We recall the following definition:
\begin{definition}(\cite[Definition 3.7]{AET})
	\label{RWMdefAET}We call a system $\mathbf{A}$ \emph{weakly mixing relative to}
	the subsystem $\mathbf{F}$ if
	\begin{equation}
	\lim_{N\rightarrow\infty}\frac{1}{N}\sum_{n=1}^{N}\lambda\left(
	|{D(a^*\alpha^{n}(a))|}^{2}\right)  =0 \label{RSV}%
	\end{equation}
	for all $a\in A$ with $D(a)=0$.
\end{definition}
Since $\mu$ is tracial, Definition \ref{RWMdefAET} coincides with \cite[Definition 3.1]{DK} 
because of \cite[Proposition 3.8]{DK}. Thus the formulation of \cite[Theorem 4.2]{DK} does not change: 

\begin{theorem}\label{thm_orig_RWM_relative_ergodicity}
    The system $\mathbf{A}$ is weakly mixing relative to
	$\mathbf{F}$ if and only if 
$\mathbf{A \odot_F A '}$ is ergodic relative to $\mathbf{F}$.
    
\end{theorem}

In the next section this theorem will allow us to formulate relative discrete spectrum in terms of the 
basic construction as a complementary concept to relative weak mixing.
\section{Relative Discrete Spectrum}
\label{afdRDS}
 The inspiration for our noncommutative definition of relative discrete spectrum is the 
treatment 
in \cite{G} of the original work of Furstenberg and Zimmer (see \cite[p. 193]{G}).
The $U$-$\bar\mu$-modules (Definition \ref{def_U_alpha_modules}) play a role analogous
to that of the finite rank modules appearing in \cite[Definition 9.2]{G} and \cite[Definition 
9.10]{G}.
However, unlike \cite{G}, we do not use an analogue of generalized eigenfunctions. Instead we opt to directly 
use 
the 
$U$-$\bar\mu$-modules to define a subspace analogous to the vector space 
$\mathcal{E}(\mathbf{X}/\mathbf{Y})$ of all 
generalized eigenfunctions appearing
in \cite[Definition 9.10]{G}.

In order to motivate our definition of relative discrete spectrum, we are
going to make use of ideas from relative weak mixing, as developed in
\cite[Sections 3 and 4]{AET}
and \cite[Section 2]{P}, and subsequently studied
further in \cite{DK} in connection to relatively independent joinings. 

We begin by defining
\[
xa:=j(a)x
\]
for all $x\in H  $ and $a\in A $, making $H  $ a right-$A $-module. Of course, $H  $
is already a left-$A $-module by $A $'s usual action on $H  $, so $H  $ is in fact a
bimodule, but it is the right module structure that will be of particular
significance for us.

\begin{definition}
Given a closed subspace $V$ of $H,$ denote the projection of $H$ onto $V$ by $P_{V}$. We call $V$ a 
\emph{right-}$F$\emph{-submodule} (of
$H$) if $VF\subset V$, i.e. if $xa\in V$ for all $x\in V$ and for all $a\in F$. 
\end{definition}

\begin{proposition}\label{prop_right_modules_basic_constr}
 Let $V$  be a closed subspace of $H$. 
    Then $V$ is a right $F$-submodule if and only if $P_V\in 
\innerproduct{A }{\e}$.
\end{proposition}
\begin{proof}
Simply note that, for all $a\in F,$
 \[
    j(F)V\subset V
    \Leftrightarrow P_V\in (JFJ)'=\innerproduct{A}{\e},
\]    
the last equality following from \cite[Lemma 4.2.3]{SS}. 
\end{proof}

We are interested in Hilbert subspaces $V$ of $H$ which are invariant under the group $\{U^n:n\in 
\mathbb{Z}\}$, therefore we say that $V$ is \emph{$U$-invariant} if 
\[UV=V,\]
rather than just assuming inclusion.
\begin{definition}\label{def_U_alpha_modules}
Suppose $V\subset H  \ominus H_{F}$ is a $U$-invariant right-$F$-submodule.
 Call $V$ a $U$\emph{-}$\bar{\mu}$\emph{-module} if in addition $V$ satisfies
\[\bar{\mu}(P_{V})<\infty.\]
\end{definition}
\begin{definition}\label{def_E(AF)}
By $\mathcal{E}_{A /F}$ denote the closed subspace of $H\ominus H_F$ spanned by all 
$U$-$\bar\mu$-modules.
\end{definition}

We now want to capture the idea that relative weak mixing and relative discrete spectrum exist as 
complementary concepts
(\cite[\S 12.4]{Tao} presents this point of view in the commutative case). It is based on the following 
result,
the one direction of which is
proven in \cite[Proposition 3.8]{AET}, although they also mention that
the other direction holds. We prove the latter using Theorem 
\ref{thm_orig_RWM_relative_ergodicity}. 

\begin{theorem}\label{thm_no_RWM}
The system $\mathbf{A }$ is weakly mixing relative to $\mathbf{F}$ if and only
if 
$\mathcal{E}_{A /F}=\{0\}.$
\end{theorem}
\begin{proof}
    Note that the statement of the theorem can be rephrased as follows: 
 The system $\mathbf{A }$ is weakly mixing relative to $\mathbf{F}$ if and only
if there are no non-trivial $U$-$\bar{\mu}$-modules.  

That \eqref{RSV} holds if 
there are no non-trivial $U$-$\bar{\mu}$-modules,  
follows from
 \cite[Proposition 3.8]{AET}.
We prove the converse as follows:

Assume there is a non-trivial $U$-$\bar\mu$-module $V$.
 Hence, $P_V\in \mathcal{N}_{\bar\mu}$ and we can set
\[x:=\gamma_{\bar\mu}(P_V)\in \bar H.\]
As $U V=V,$ we have $\bar\alpha(P_V)=U P_V U^{\ast}=P_V.$
Hence, $x\in \bar{H}^{\bar{U}},$
 with $x\neq 0,$ since $P_V\neq 0$ 
and $\bar\mu$ is faithful.

Since $P_V\e=0,$ 
\[\innerproduct{x}{\gamma_{\bar\mu}(\e a)}_{\bar\mu}=\bar\mu(P_V^{\ast}\e a)=0,\]
for all $a\in F.$
Hence, from \eqref{eq_eF}, $x\perp \bar{H}_\lambda,$ so $x\notin \bar{H}_\lambda$ (since $x\neq 
0$) and thus
$\bar{H}^{\bar{U}}\not\subset \bar{H}_\lambda.$

In other words, $\mathbf{A \odot_F A '}$ is not ergodic relative to $\mathbf{F}$. By
Theorem \ref{thm_orig_RWM_relative_ergodicity} we are done.
\end{proof}

Motivated by this result, we now present the main definition of this paper:
\begin{definition}\label{defRDS}
  We say that the system $\mathbf{A }$ has \emph{discrete spectrum relative to} 
$\mathbf{F}$ if $\mathcal{E}_{A /F}=H  \ominus H_F.$ Alternative terminology for this is to say 
that $\mathbf{A }$ is an \emph{isometric extension} of $\mathbf{F}$.
\end{definition}
Thus relative weak mixing and relative discrete spectrum correspond to the two extremes of 
$\mathcal{E}_{A/F}$, and are, in this sense, complementary.

 In the remainder of this section we
show that the classical definition of relative discrete spectrum as well as
the absolute case of noncommutative discrete spectrum are special cases of
this definition, confirming that it is a sensible definition in a
noncommutative framework. 

The classical notion of relative discrete spectrum is  
defined as follows (see \cite[Definition 
9.10]{G}):
\begin{definition}\label{def_rds_classical}
Assume that $\mathbf{A }$ is a classical system, i.e. $A =L^{\infty}(\eta)$ for
a standard probability space $(Y,\Sigma,\eta)$. A $F$-submodule $V$ of
$H  =L^{2}(\eta)$ is said to be of \emph{finite rank}
if there are $x_{1},...,x_{n}\in
V$ such that 
\[V=\overline{\left\{\sum^n_{i=1} a_i x_i:a_{1},...,a_{n}\in F\right\}},\]
where
$a_{j}x_{j}$ is simply pointwise multiplication of functions. We call $x\in H  $
an $F$\emph{-eigenvector} of $U$ if $x$ belongs to some $U$-invariant finite
rank $F$-module (for simplicity, $x=0$ is allowed). If $H\ominus H_F$ is spanned by the $F$-eigenvectors of 
$U$, then we say
that $\mathbf{A }$ has \emph{relative discrete spectrum over $\mathbf{F}$ in the
classical sense.}
\end{definition}

\begin{remark}
	In \cite{G}, the condition that $H\ominus H_F$ is spanned by the $F$-eigenvectors of $U$, is expressed as 
$H$ being spanned by the $F$-eigenvectors of $U$. These two conditions are equivalent. 
	This is simply because $H_F$ is a finite rank $U$-invariant $F$-module. Hence all elements of $H_F$ are 
$F$-eigenvectors of $U$, so if $x\in H$ is an $F$-eigenvector of $U$, then so is $e_F x\in H_F$, and therefore 
$(1-e_F)x\in H\ominus H_F$ as well.
\end{remark}

Definition \ref{def_rds_classical} is indeed a special case of Definition \ref{defRDS} as is proved below in 
Proposition 
\ref{prop_RDSiffRDSclassical}. The proof uses direct integral theory, as 
it is used in \cite[Lemma 4.1]{AET}. This is why we assume that $(X,\mathcal{X},\eta)$ be 
standard, as it ensures 
that $L^2(\eta)$ is separable (\cite[Corollary 5.3]{nielsen1980}). 

\begin{proposition}\label{prop_RDSiffRDSclassical}
Assume that $\mathbf{A }$ is a classical system, 
i.e. $A =L^{\infty}(\eta)$ for
a standard probability space $(X,\mathcal{X},\eta)$ and $\alpha(f)=f\circ T$ for some fixed invertible map 
$T:X\rightarrow X$ satisfying $\eta(Z)=\eta(T^{-1}(Z))$ for all $Z\in \mathcal{X}.$ 
The system $\mathbf{A }$ 
has discrete
spectrum relative to $\mathbf{F}$ (in the sense of Definition \ref{defRDS})
if and only if it has relative discrete spectrum over $\mathbf{F}$ in the classical
sense. 
\end{proposition}
\begin{proof}
Assume that $\mathbf{A }$ has discrete spectrum relative to $\mathbf{F}.$ The approach of the proof is to 
express any $U$-$\bar\mu$ module 
$V$ as 
the direct sum of finite rank modules, using ideas from the proof of \cite[Lemma 4.1]{AET}. 

Using \cite[Theorem 14.2.1]{KadRing2}, since $F$ is commutative, we have a unitary operator 
$\Phi:H  \rightarrow H_\oplus$ 
where $H_\oplus$ is a direct 
integral $H_\oplus=\int_Y^\oplus H_p \romd\nu(p)$ of Hilbert spaces $H_p$ indexed by some standard 
probability 
space $(Y,\mathcal{Y},\nu).$ Thus, in particular, any statement about a module $V$ in 
$H_\oplus$ has a corresponding statement about $\Phi^{-1}V$ in $H$. 

Define 
\[\phi:F\rightarrow \mathcal{B}(H_\oplus):a\mapsto \Phi a \Phi^{-1}.\]
The von Neumann algebra $F$ is then identified with the von Neumann algebra of all 
diagonalizable operators $\phi(F)=\{M_f: f\in L^\infty(\nu)\}$ where $M_f\in 
\mathcal{B}(H_\oplus)$ is the multiplication operator acting 
on $x\in H_\oplus$ via the equality 
$(M_fx)(p)=f(p)x(p)$ for 
almost all $p\in X$.
 Given any $U$-$\bar\mu$-module $V,$ then as in the proof of \cite[Lemma 4.1]{AET} we can write 
 \[\Phi V=\int^\oplus_Y V_p\romd\nu(p),\] 
 for a measurable field of Hilbert subspaces $V_p\subset H_p.$
 
We shall now express $\Phi V$ as a 
direct sum of $\phi(F)$-modules of finite rank. 
For 
each $n\in 
\mathbb{N}\cup \{\infty\}$ write
\[Y_n:=\{p\in Y: \mathrm{dim}\,(H_p)=n\}.\] 

Each $Y_n$ turns out to be measurable \cite[Remark 14.1.5]{KadRing2}. 
Consider the projections $M_{\chi_{Y_n}}$
and define
\[V_n:=\int_{Y_n} V_p  \romd\nu(p)=M_{\chi_{Y_n}}\Phi V,\]
where $\chi_{Y_n}$ denote the indicator functions. As in the proof of \cite[Lemma 4.1]{AET}, $\int_Y 
\mathrm{dim}(V_p)\romd\nu(p)<\infty$, so $\nu(Y_\infty)=0,$ hence $V_\infty=0$ 
and the collection $\{Y_n: n\in 
\mathbb{N}\}$ satisfies $\nu(\cup_{n\in \mathbb{N}}Y_n)=1.$ 
It follows that $\Phi V$ can be identified with $\oplus_{n\geq 1} V_n.$

It is now straightforward to verify that each $\Phi^{-1}V_n$ is a $U$-$\bar\mu$-module:
We have, for every $f\in F,$
\[f\Phi^{-1}V_n=f\phi^{-1}(M_{\chi_{Y_n}})(V)=\phi^{-1}(M_{\chi_{Y_n}})fV\subset 
\phi^{-1}(M_{\chi_{Y_n}})V=\Phi^{-1}V_n,\]
so that each $V_n$ is a right $\phi(F)$-module. 

In a similar way to the proof of \cite[Lemma 4.1]{AET}, $\alpha$ induces dynamics on $Y$ leaving each $Y_n$ 
invariant, which in turn means that each $V_n$ is $U$-invariant, since $\Phi U \Phi^{-1}$ is given by a 
measurable section of unitary operators $\Psi:Y\rightarrow \amalg_{p\in Y}\mathcal{U}(H_p)$ combined 
with $S$.

By construction, $\mathrm{dim}(V_p )\leq n$ 
whenever $p\in Y_n$ and it follows that $\Phi^{-1}V_n$ is of finite rank.

So $\Phi V$ consists solely of $\phi(F)$-eigenvectors  
and hence $V$ and therefore (because of Definitions \ref{defRDS} and \ref{def_E(AF)}) also $H  \ominus H_F$ 
are 
spanned by $F$-eigenvectors as required.

We now prove the converse. Assume that $\mathbf{A}$ has relative discrete spectrum over 
$\mathbf{F}$ in the 
classical sense.
 Then we simply have to show that the projection $P_V$ corresponding to a finite rank 
$F$-module $V\subset H\ominus H_F$
satisfies $\bar\mu(P_V)<\infty.$
 
Consider then any finite rank $F$-module $V:=\overline{\, \left\{\sum^{n}_{i=1} f_iv_i: f_i\in 
F\right\}}.$ 

We now give a description of $V_p$ for almost all $p .$
Put $w_i:=\Phi v_i$ for each $i=1,2,\ldots,n.$ 
Thus, 
\[\Phi V=\overline{\left\{\sum^n_{i=1} M_{g_i}w_i: g_i\in 
L^\infty(\nu)\right\}}.\] 

Hence all vectors of the form $M_gw$ for $g\in L^\infty(\nu)$ and $w\in \{w_i: i=1,2,\ldots,n\}$ 
form a 
dense spanning set
for $\Phi V$ and thus, from \cite[Lemma 14.1.3]{KadRing2}, for almost all $p,$
\begin{align}
  V_p
            &=\overline{\left\{\sum^n_{i=1} g_i(p)w_i(p): g_i\in 
                L^\infty(\nu)\right\}}\nonumber\\
            &=\mathrm{span}\{w_i(p): i=1,2,\ldots,n\}\nonumber.
\end{align}

 Similar to the proof of \cite[Lemma 4.1]{AET}, we thus have,
\[ \bar\mu(P_V)=\int_Y \mathrm{dim}(V_p) \romd\nu(p))\leq \int_Y n 
\romd\nu(p)=n<\infty.\]
\end{proof}

We consider another special case of Definition \ref{defRDS} when $F=\mathbb{C}1$ and 
$\lambda=\mu|_F.$

 We take note that in this case the basic construction is given by 
$\innerproduct{A }{\e}=JF'J=J\mathcal{B}(H)J=\mathcal{B}(H),$ using \cite[Lemma 4.2.3]{SS}. Thus, since the 
trace on $\mathcal{B}(H)$ is unique up to nonzero scalar multiples, we may take $\bar\mu$ to be the 
canonical trace $\mathrm{Tr}$ on $\mathcal{B}(H).$ In particular, this means that our $U$-$\bar\mu$-modules 
are exactly the finite dimensional $U$-invariant subspaces of $H.$
\begin{proposition}
 Let $\mathbf{A}=(A,\mu,\alpha)$ be a system and $\mathbf{F}$ be the trivial system i.e 
$F=\mathbb{C}1,$
 $\lambda=\mu|_F,$ and $\varphi=\alpha|_{F}.$ Then $\mathbf{A}$ has discrete spectrum relative to 
$\mathbf{F}$ if and only if $\mathbf{A}$ has discrete spectrum, i.e $H$ is spanned by 
the eigenvectors of $U.$
\end{proposition}
\begin{proof}
 Note that $\Omega$ is always a fixed point of $U$. Let $\mathcal{E}$ denote the set of all eigenvectors of 
$U$ orthogonal to $\Omega$. Assume that 
$\mathbf{A}$ has discrete spectrum, i.e. $\overline{\mathrm{span}\, \mathcal{E}}=H\ominus \mathbb{C}\Omega.$
 For $x\in \mathcal{E},$ 
 let 
 \[S_x:=\{sx: s\in \mathbb{C}\}.\]
 Then, it easy to verify that $S_x$ is a $U$-$\bar\mu$-module.
 Moreover,
 \[H\ominus H_F=\overline{\mathrm{span}\{S_x: x\in \mathcal{E}\}}.\]
Thus, $\mathbf{A}$ has discrete spectrum relative to $\mathbf{F}.$

Conversely, assume that $\mathbf{A}$ has discrete spectrum relative to $\mathbf{F}.$ Then, as remarked above,
all $U$-$\bar\mu$-modules $V$ have finite dimension, and they 
span $H\ominus \mathbb{C}\Omega$. As each such finite dimensional $U$-invariant space $V$
is spanned by eigenvectors of $U$, $H\ominus \mathbb{C}\Omega$ is as well. It follows that $\mathbf{A}$ has 
discrete spectrum.
\end{proof}

\section{Skew Products}\label{afdEU}
In order to show that the definition of relative discrete spectrum
(Definition \ref{defRDS}) is sensible, we still need to exhibit some examples. This is what we do in 
this 
section and the next.

In this section we focus on a skew product, starting with a classical system and extending it by a 
noncommutative one.

The following result will be useful for both examples:
\begin{proposition}\label{prop_basic_constr}
 Let $(B ,\nu)$ and $(C,\sigma)$ be von Neumann algebras with faithful normal tracial states $\nu$ 
and 
$\sigma,$ 
both in their GNS representations on the Hilbert spaces $H_\nu$ and $H_\sigma,$ with cyclic vectors 
$\Omega_\nu$ and $\Omega_\sigma$, respectively. Consider the von Neumann algebra tensor product $A 
:=B \bar\otimes C$ 
and the faithful normal state $\mu:=\nu\bar\otimes \sigma$.
Set $F:=B \otimes 1$ with state $\lambda:=\mu|_F.$ Then 
\[\innerproduct{A }{\e}=B \bar\otimes 
\mathcal{B}(H_\sigma).\]
The trace $\bar\mu$ of $\innerproduct{A }{\e}$ is given by 
\begin{equation}\label{eq_skew_product_state_normality}
\bar\mu(t)=\sum_{i\in \mathcal{I}} \innerproduct{\Omega_\nu\otimes h_i}{t(\Omega_\nu\otimes  h_i)}=\mu\bar\otimes \mathrm{Tr}(t),
\end{equation}
for all $t\in \innerproduct{A }{\e}^+,$ where $\{h_i:i\in \mathcal{I}\}$ is any orthonormal basis for 
$H_\sigma$ 
 and $\mathrm{Tr}$ is the canonical trace  on $\mathcal{B}(H_\sigma)$.
\end{proposition}
\begin{proof}
 Let $J_\nu$, $J_\sigma$ and $J=J_\nu\otimes J_\sigma$ denote the modular conjugation operators 
associated to $\nu$, $\sigma$ and 
$\mu$,
respectively. 
By \cite[Lemma 4.2.3]{SS} 
and \cite[Section 10.7 Lemma 1]{stratila1979lectures} we have
\begin{equation}\label{eq_tp_calc}
 \innerproduct{A }{\e}
 =JF'J 
 =(J_\nu B' J_\nu)\bar\otimes (J_\sigma \mathcal{B}(H_\sigma)J_\sigma)
 =B \bar\otimes\mathcal{B}(H_\sigma).
\end{equation}

We compute the trace $\bar\mu$ using \cite[Lemma 4.3.4]{SS}. To do this, we need elements $v_i$ 
of $\innerproduct{A '}{\e}$ for $i\in \mathcal{I}$ such that $\sum_{i\in \mathcal{I}} v^*_i\e 
v_i=1$ (see Remark \ref{rem_net_vi} below).
Let 
\[v_i=1\otimes w_i\]
where, for all $z\in H_\sigma,$ $w_i\in \mathcal{B}(H_\sigma)$ is defined by 
\[w_iz:=\innerproduct{J_\sigma h_i}{z}\Omega_\sigma.\]

Note that,
\begin{align}
 \innerproduct{A '}{\e}&=\innerproduct{JA  J}{J\e J}=J\innerproduct{A }{\e}J\nonumber\\
                       &=(J_\nu B  J_\nu)\bar\otimes (J_\sigma B  (H_\sigma)J_\sigma) \nonumber\\
					   &=B'\bar\otimes \mathcal{B}(H_\sigma).\nonumber
\end{align}
So we have $v_i\in \innerproduct{A '}{\e}$.

In terms of the projection $P$ of $H_\sigma$ onto 
$\mathbb{C}\Omega_\sigma$ we have $\e=1\otimes P$, since $H=H_\nu\otimes H_\sigma$
and $H_F=H_\nu\otimes (\mathbb{C}\Omega_\sigma)$. Hence
\begin{align}
 v_i^{\ast}\e v_i&=1\otimes w^{\ast}_iP w_i.\nonumber
\end{align}

For each $i,$ the linear operator $w^{\ast}_i Pw_i$ is the projection of $H_\sigma$ onto 
$\mathbb{C}J_\sigma h_i.$
Hence, 
\begin{equation}\label{eq_sum_par_isom}
 \sum_{i\in \mathcal{I}} v^{\ast}_i\e v_i=1.
\end{equation}

Thus, applying the formula in \cite[Lemma 4.3.4]{SS} in terms of $\Omega=\Omega_\nu\otimes\Omega_\sigma$, for 
all $t\in \innerproduct{A }{\e}^+$,
\begin{align}
 \bar\mu(t)&=\sum_{i\in \mathcal{I}}\innerproduct{Jv^{\ast}_i\Omega}{tJv^{\ast}_i\Omega}\nonumber\\
       	   &=\sum_{i\in \mathcal{I}} \innerproduct{\Omega_\nu\otimes h_i}{t(\Omega_\nu\otimes 
 h_i)}.\nonumber
\end{align}
Since $\bar\mu$ is faithful and the first equality of \eqref{eq_skew_product_state_normality} 
holds, it follows from  
 \cite[Theorem 8.2]{S} that the second equality of \eqref{eq_skew_product_state_normality} holds.
\end{proof}
\begin{remark}\label{rem_net_vi}
 \cite[Lemma 4.3.4]{SS}  
 requires a net $(v_i)$ satisfying \eqref{eq_sum_par_isom}. However,  
 the assumption that $\mathcal{I}$ is a directed set is not used, neither in the proof of \cite[Lemma 
4.3.4]{SS} 
nor in any results that \cite[Lemma 4.3.4]{SS} 
depends on.  
\end{remark}

We now turn to the skew product.
Let $(X,\mathcal{X},\rho)$ be a standard probability space with compact Hausdorff space $X$ and 
Borel measure 
$\rho.$ 
We let $S:X\rightarrow X$ be an invertible map such that $S^{-1}\mathcal{X}\subset \mathcal{X}$ and 
$S\mathcal{X}\subset \mathcal{X}$, and which is measure preserving with respect to $\rho,$ 
that is,
\[  \rho(K)=\rho(S^{-1}(K)),\] 
for all $K\in \mathcal{X}$.

We set 
\[B:=L^\infty(\rho),\, \Omega_\nu:=1,\, \nu(f):=\int_X f\romd\rho\text{ and }\beta:B\rightarrow B:f\mapsto f\circ 
S.\] 
Then 
$\mathbf{B}$ is a system if we view $B $ as operators acting on $L^2(\rho)$ via pointwise multiplication: for 
every $f\in 
L^\infty(\rho),$ we have an operator 
\[M_f:L^2(\rho)\rightarrow L^2(\rho):g\mapsto fg.\]

We let 
\[\mathbf{C}=(C,\sigma,\gamma)\] 
be a system such that $H_\sigma$ in Proposition \ref{prop_basic_constr} is separable. Denote the unitary 
representation of $\gamma$ on $H_\sigma$ by
$U_\gamma$.

Now put
\[A :=B \bar\otimes C.\] Then \[(L^2(\rho)\otimes H_\sigma, \mathrm{id}_A , 1\otimes \Omega_\sigma)\] 
is 
the GNS triple for $A $ associated to the product state 
\[\mu:=\nu\bar\otimes\sigma.\] Put  
\[F:=B \otimes 1\] 
and let $\lambda:= \mu|_{F}.$

We construct the skew product dynamics $\alpha$ on $A $ using the theory of direct integrals (see 
for 
example \cite{nielsen1980} and \cite[Section IV.8]{Tak1}). 
Consider the space of $H_\sigma$-valued $\rho$-square integrable functions $L^2(\rho;H_\sigma)$. 
Then $L^\infty(\rho)$ is $\ast$-isomorphic to the von Neumann algebra 
$\mathcal{M}$ of all
diagonalizable operators on $L^2(\rho; H_\sigma)\cong L^2(\rho)\otimes H_\sigma $ (\cite[Proposition 
5.2]{nielsen1980}). 
In effect, any $f\in L^\infty(\rho)$
is identified with $M_f\otimes 1$. Furthermore,  $1\otimes \Omega_\sigma$ is represented by
$\Omega\in L^2(\rho, H_\sigma)$ given by $\Omega(p)=\Omega_\sigma$ for all $p\in X$.
If we put $\mathcal{N}(p)=C$ for all $p\in X,$ then
from \cite[Corollary 19.9]{nielsen1980} 
 and its proof
we have the isomorphism
\[\int^\oplus_X C \romd\rho(p):=\int^\oplus_X \mathcal{N}(p)\romd\rho(p) \cong B 
\bar\otimes C.\]
We identify $A=B\bar\otimes C$ with this integral in the remainder of this section.
The elements $a=\int^\oplus_X a(p) \romd\rho(p)$ of $\int^\oplus_X C \romd\rho$ 
consist of decomposable 
operators with $a(p)\in \mathcal{B}(H_\sigma)$ for all $p\in X$, such that

\[\norm{a(\cdot)}\in L^\infty(\rho),\]
and for any $z\in L^2(\rho;H_\sigma)$ the element $az\in L^2(\rho, H_\sigma)$ is given by 
\[(az)(p)=a(p)z(p)\]
for all $p\in X$.
Moreover, from \cite[Theorem IV.8.18]{Tak1}, we have $ a(p)\in C$. Thus, we may represent each $a\in 
\int^\oplus_X C \romd\rho$ by a map $a:X\rightarrow C:p\mapsto a(p).$ In particular, $a=b\otimes c\in A$ is 
given by $a(p)=b(p)c,$ for any $b\in B=L^\infty(\nu)$ and $c\in C$.

Let \[k:X\rightarrow \mathbb{Z}\] be any measurable map.
For $a\in \int^\oplus_X C \romd\rho,$ define for all $p\in X,$
\begin{equation}\label{eq_dyn_skew}
 \alpha(a)(p):=\gamma^{k(p)}(a(Sp)).
\end{equation}
Then $\alpha$ is the skew product dynamics, where $k$ acts as the generator of a cocycle.
It is straightforward to verify that $\alpha$ is a well-defined $\ast$-automorphism of $A$ leaving 
$\mu$ invariant, i.e. that $\mathbf{A}=(A,\mu,\alpha)$ is a system.

Notice that $F$ is invariant under $\varphi=\alpha|_F$, since for all $p\in X,$
\begin{align}
 \alpha(b\otimes 1)(p)=(b\circ S)\otimes 1.
\end{align}

We describe the unitary representation $U$ of $\alpha$.
Note first that
\begin{align}
(Ua\Omega)(p)&=(\alpha(a)\Omega)(p)=\alpha(a)(p)\Omega(p)=\gamma^{k(p)}
(a(Sp))\Omega_\sigma\nonumber\\
&=U^{k(p)}_\gamma(a(Sp)\Omega_\sigma)=U^{k(p)}_\gamma(a\Omega)(Sp).\nonumber
\end{align}
Let $x\in \int^\oplus_X H_\sigma \romd\rho(p)$ and approximate $x$ by a sequence 
$(x_n)=(a_n\Omega)$ in $A \Omega$.

Since, 
\begin{align}
 \int_X \norm{x_n(Sp)-x(Sp)}^2\romd\rho(p)&=\norm{x_n-x}^2
		\rightarrow 0\quad\text{ as } n\rightarrow 
\infty,\nonumber
\end{align}
it follows as in the proof of the completeness of $L^p$ spaces, that
there is a subsequence $(\norm{x_{n_i}(Sp)-x(Sp)})$ which tends to $0$ except for $p$ in a null set
$N_0\subset X$. 

Thus,
\[(Ux)(p)=\lim_i U^{k(p)}_\gamma x_{n_i}(Sp)=U^{k(p)}_\gamma x(Sp),\]
for all $p\in X\backslash N_0$. Without loss, we may define $Ux$ such that this holds for 
all $p\in X$. Then it follows that
\begin{equation}\label{eq_inverse_Uskewprod}
 (U^{-1}x)(p)=U_\gamma^{-k(S^{-1}p)}x(S^{-1}p).
\end{equation}

To conclude, we discuss a concrete example of $C.$ The main points from this example are summarized 
in Proposition
\ref{prop_summary_skew_product}.
\begin{example}\label{ex_skew_product}
Let $G$ be a countable group endowed with the discrete topology and let $T:G\rightarrow 
G$ be any group automorphism such that for 
each $g\in 
G$ the orbit of $g,$ $T^\mathbb{Z}g:=\{T^ng: n\in \mathbb{Z}\}$, is a finite set (we refer to 
$T^\mathbb{Z}g$ as a \emph{finite orbit}). Consider the dual system on 
\[C:=\mathfrak{L}(G),\] the group von 
Neumann 
algebra of $G.$ 
 Thus, $C$ is the von Neumann algebra on
$\ell^2(G)$ generated by the following set of unitary operators:
\begin{equation}
\{l(g):g\in G\}
\end{equation}
where $l$ is the left regular representation of $G$, i.e. the unitary
representation of $G$ on $\ell^2(G)$ with each $l(g):\ell^2(G)\rightarrow \ell^2(G)$ given by
\[
\lbrack l(g)f](h)=f(g^{-1}h)
\]
for all $f\in \ell^2(G)$ and $g,h\in G$. Equivalently,
\[
l(g)\delta_{h}=\delta_{gh}
\]
for all $g,h\in G$, where $\delta_{g}\in \ell^2(G)$ is defined by $\delta_{g}(g)=1$
and $\delta_{g}(h)=0$ for $h\neq g$. Setting
\[
\Omega_\sigma:=\delta_{1}
\]
where $1\in G$ denotes the identity of $G$, we can define a faithful normal
trace $\sigma$ on $B$ by
\[
\sigma(a):=\innerproduct{\Omega_\sigma}{a\Omega_\sigma} 
\]
for all $a\in C$. 
 It follows that $(\ell^2(G),\mathrm{id}_{C},\Omega_\sigma)$ is the cyclic
representation of $(C,\sigma)$.

We have a unitary $U_{\gamma}:\ell^2(G)\rightarrow \ell^2(G),$ defined by
\[U_\gamma(f)=f\circ T.\]

We define a $*$-automorphism $\gamma$ on $C$ by
\[\gamma(c)=U_\gamma c U^{\ast}_\gamma,\]
for all $c\in C$. Then, $(C, \sigma, \gamma)$ is a system.

Using Proposition \ref{prop_basic_constr}, the basic construction is given  
by \[\innerproduct{A }{\e}=L^\infty(\rho) \bar\otimes \mathcal{B}(\ell^2(G)).\]

For each $g\in G$ let
\[R_g:=\mathrm{span}\,(U^\mathbb{Z}_\gamma \delta_g)\]
and let $Q_g$ be the projection of $\ell^2(G)$ onto $R_g$.
Set 
\[V_g:=L^2(\rho)\otimes R_g\]
and let $P_g=1\otimes Q_g$ be the projection of 
$H:=L^2(\rho)\otimes \ell^2(G)$ onto $V_g$.

We have
\begin{align}
\bar\mu(P_g)=\sum_{h\in G} \innerproduct
{\Omega_\nu\otimes \delta_h}{P_g(\Omega_\nu\otimes\delta_h)} 
=\sum_{h\in G}\innerproduct{\delta_h}{Q_g\delta_h}=\mathrm{dim}(R_g)<\infty,\nonumber
\end{align}
since all orbits are finite.

The $V_g$'s , for $g\neq 1$, span $H\ominus H_F=L^2(\rho)\otimes \Omega^\perp_\sigma$, since the 
$R_g$'s span $\Omega^\perp_\sigma$.
As $R_g$ is spanned by an orbit, we have $U_\gamma R_g=R_g$. It follows that if
$x\otimes y\in V_g$, then,
\begin{align}U(x\otimes y)(p)=U^{k(p)}_\gamma(x\otimes y)(Sp)
=U^{k(p)}_\gamma(x(Sp) y)=x(Sp)U^{k(p)}_\gamma y\in R_g,\nonumber
\end{align}
for all $p\in X$, since $x\otimes y$ is represented by 
$p\mapsto x(p)y$ in $\int^\oplus_X H_\sigma\romd(\rho)$. Hence $U(x\otimes y)\in L^2(\rho)\otimes 
R_g$, so  $UV_g\subset V_g$. Using \eqref{eq_inverse_Uskewprod}, it similarly follows that
$U^{-1}V_g\subset V_g$, so $UV_g=V_g$.

The $V_g$'s are trivially right-$F$-modules, since $F=L^\infty(\rho)\otimes 1$. Hence
the $V_g$'s are indeed $U$-$\bar\mu$-modules which (when excluding $g=1$) span
$H\ominus H_F$ as required by Definition \ref{defRDS}.
\end{example}

We briefly summarize:
\begin{proposition}\label{prop_summary_skew_product}Consider a dual system $\mathbf{C}$ generated from a 
discrete countable group 
$G$, with automorphism $T:G\rightarrow G$ with finite orbits, and a 
classical system $\mathbf{B}$ obtained from a standard measure-preserving 
system $(X,\mathcal{X},\rho, S)$.
Form the system $(B\bar\otimes C, \mu, \alpha)$ with $\mu$ as a vector state from $1\otimes \delta_1$ 
and dynamics given by equation
\eqref{eq_dyn_skew}. Then $(B\bar\otimes C, \mu, \alpha)$ has discrete spectrum relative to $(B\otimes 
1, \mu|_{B\otimes 
1}, 
\alpha|_{B\otimes 
1}).$

\end{proposition}

Taking $G$ to be the free group on a finite or countable set of symbols, with $T$ induced by a 
finite orbit bijection of the symbols,  
provides a concrete and non-trivial realization of $C$.
\section{Finite Extensions}\label{sec_finite_ext}

In this section we present a second example of relative discrete spectrum. In this case, unlike the 
previous section, we start with a noncommutative system and extend it by a finite dimensional 
noncommutative system (hence the name ``finite extension'').

Let $M_n=M_n(\mathbb{C})$ denote the 
$n\times n$ matrices over $\mathbb{C}$.
\begin{definition}\label{def_finite_ext}
Consider a system $\mathbf{B}=(B,\nu, \beta).$ Let $n\in \mathbb{N}$.  
Consider the von Neumann algebra
$A =B\odot 
M_{n}$ with 
faithful normal trace $\mu=\nu\odot  \mathrm{tr},$ where $\mathrm{tr}$ is the normalized 
trace on 
$M_n$.  Suppose 
further that there is a 
$*$-automorphism $\alpha$ of $A $ such that $\alpha(b\otimes 1)=\beta(b)\otimes 1.$  Represent $\mathbf{B}$ 
as the subsystem $\mathbf{F}$ of $\mathbf{A}$ given by $F=B\otimes 1$, $\lambda(b\otimes 1)=\nu(b)$ and 
$\varphi(b\otimes 1)=\beta(b)\otimes 1$.
Then we refer to $\mathbf{A }=(A ,\mu, 
\alpha)$ as 
a \emph{finite extension of} $\mathbf{F}$. Equivalently, we say that $\mathbf{A}$ is a \emph{finite 
extension} of $\mathbf{B}$.
\end{definition}

Note that we can  
view $B \odot M_n$ as all $n\times n$ matrices with entries in $B$. 

There is a general reason why finite extensions are isometric extensions (Proposition 
\ref{prop_finite_ext}): if the trace on the basic construction is finite, we automatically have relative discrete spectrum, 
as we now show (Corollary \ref{cor_fin_auto_rds}).

\begin{proposition}\label{prop_HA -HF}Let $\mathbf{A }$ be a system with subsystem $\mathbf{F}.$
 Then the subspace $H  \ominus H_F$ is a $U$-invariant right $F$-submodule.
 \end{proposition}
 \begin{proof}
  Consider $H  \ominus H_F$ and its corresponding projection $1_A -\e.$ Since $1_A -\e\in 
\innerproduct{A }{\e},$ $H  \ominus H_F$ is a right 
$F$-module using Proposition \ref{prop_right_modules_basic_constr}. Furthermore, since 
$\alpha(F)=F,$ we have $U^{\ast} H_F=H_F.$ Consequently, for $x\in H  \ominus H_F$  and $y\in 
H_F,$ we have 
\begin{equation}\label{eq_inv_HA -HF}
\innerproduct{U x}{y}=\innerproduct{x}{U^{\ast} y}=0,
\end{equation}
so that $U(H  \ominus H_F)\subset H  \ominus H_F.$ Similarly, we have $U^{\ast}(H  \ominus H_F)\subset H  
\ominus H_F.$
 \end{proof}

\begin{corollary}\label{cor_fin_auto_rds}
Suppose that $\mathbf{A }$ is a system with subsystem $\mathbf{F}$ and assume that 
$\bar\mu$  is finite, in the sense that $\bar\mu(x)<\infty$, for every $x\in \innerproduct{A }{\e}^+.$ 
Then 
$\mathbf{A }$ has discrete spectrum relative to $\mathbf{F}.$
\end{corollary}
\begin{proof}
 Since $\bar\mu(1_A -\e )<\infty,$  $H  \ominus H_F$ is spanned by a $U$-$\bar\mu$-module, 
namely
itself. 
\end{proof}

Since the basic construction of a finite dimensional von Neumann algebra is again 
finite dimensional, the 
trace on the basic construction is finite and we have:
\begin{corollary}
 Every system on a finite dimensional von Neumann 
algebra has discrete spectrum relative to every subsystem.
\end{corollary}

Another example follows from \cite[Proposition 
3.1.2]{JonSun1997}:
\begin{corollary}
    Suppose that both $A $ and $F$ are type $\mathrm{II}_1$ factors and that
    their index $[A :F]$ is finite. Then $\mathbf{A }$ has discrete spectrum relative to 
$\mathbf{F}.$
\end{corollary}

Using Corollary \ref{cor_fin_auto_rds}, we can also prove the following:

\begin{proposition}\label{prop_finite_ext}
 If $\mathbf{A }$ is a finite extension of $\mathbf{F},$ then $\mathbf{A }$ has discrete spectrum 
relative to 
$\mathbf{F}.$
\end{proposition}
\begin{proof}
Without loss of generality, assume that $(B,\nu)$ in Definition \ref{def_finite_ext} is in its GNS 
representation $B\rightarrow 
\mathcal{B}(H_\nu)$ 
with cyclic vector $\Omega_\nu.$
 One can easily verify that the GNS triple for $M_n$ is $(\mathbb{C}^n\odot 
\mathbb{C}^n, \pi_n, \Lambda),$ where
 $\pi_n:M_n\rightarrow M_n\odot M_n:c\mapsto c\otimes 1,$ 
and
 $\Lambda=\frac{1}{\sqrt{n}}\sum_{j=1}^n e_j\otimes e_j$ with $\{e_j\}$ an orthonormal basis for 
$\mathbb{C}^n.$
Thus the GNS triple for $A =B\odot M_{n}$ is given by $(H_\nu\odot 
\mathbb{C}^n\odot\mathbb{C}^n,\pi,\Omega),$
where $\Omega=\Omega_\nu\otimes \Lambda$ and $\pi:B \odot M_n\rightarrow 
B\odot 
M_n\odot 
M_n:a\mapsto a\otimes 1$.

From Proposition \ref{prop_basic_constr},
 \[\innerproduct{A }{\e}=B \odot M_n\odot M_n\]
and 
\[\bar\mu=\nu\odot \mathrm{Tr},\]
where $\mathrm{Tr}:=\mathrm{Tr}_n\odot\mathrm{Tr}_n$, with 
$\mathrm{Tr}_n$ the usual trace (sum of diagonal entries) on $M_n$.

As $\bar\mu$ is finite, $\mathbf{A }$ has discrete spectrum relative to $\mathbf{F},$ by Corollary 
\ref{cor_fin_auto_rds}.
\end{proof}

\begin{example}\label{ex_finite_ext}
 We give a concrete realization of a finite extension for which the dynamics is not compact nor a 
tensor product of the dynamics on the underlying algebras. For simplicity, we focus on the case 
$n=2$ in Definition \ref{def_finite_ext}.
 
We let $\mathbf{B_1}=(B_1, \nu_1, \beta_1)$  and $ \mathbf{B_2}=(B_2,\nu_2, \beta_2)$ be 
systems.

Consider $B=B_1\oplus B_2$ which we view as the set of all matrices of the form 
\[\begin{bmatrix} b_1 &0 \\ 0& 
b_2\\ \end{bmatrix}\]
for $b_1\in B_1$ and $b_2\in B_2.$
Let $s\in (0,1)\subset \mathbb{R}$ and put 
\[\nu=s(\nu_1\oplus 0)+(1-s)(0\oplus \nu_2).\] 
Then $\nu$ is a faithful normal state 
on $B.$ 
So $\mathbf{B}=(B,\nu,\beta),$ with $\beta=\beta_1\oplus \beta_2,$ is a system.
 
  Set 
  \[A =B \odot M_2 \qquad\text{and}\qquad \mu=\nu\odot \mathrm{tr}.\] 
  We now describe dynamics on $(A ,\mu)$.
 Let  
\[W=\begin{bmatrix}
     w_1 &w_2\\
     w_3 &w_4\\
    \end{bmatrix}\in A ,
\]
be unitary, where $w_i\in B,$ and define $\alpha(a):=WaW^\ast$ for all $a\in B\odot M_2$.
Then $\mathbf{A}=(A,\mu,\alpha)$ is a system.
  
 From direct calculations, the requirements that $W$ satisfy 
 $\alpha(b\otimes 1)=W\begin{bmatrix} b & 0\\ 0& b\end{bmatrix}W^{\ast}\in B\otimes 1$ for every 
$b\in B,$ and 
that 
$\alpha(b\otimes 1)=\beta(b)\otimes 1$, yield
\begin{equation}\label{eq_beta_fe_eg}
 \beta(b)=w_1bw^{\ast}_1+w_2bw^{\ast}_2=w_3bw^{\ast}_3+w_4bw^{\ast}_4
 \end{equation}
and  
\[w_1bw^*_3+w_2bw^*_4=w_3bw^*_1+w_4bw^*_2=0\]
for all $b\in B$. The direct sum structure of $B$ will allow us to satisfy the latter requirement easily, 
while still giving nontrivial dynamics. This is done by setting
 \[w_1=v_1\oplus 0 \qquad\text{and }\qquad w_4=v_4\oplus 0\] 
 for $v_1,v_4\in B_1$,
 and 
 \[w_2=0\oplus v_2 \qquad\text{and }\qquad w_3=0\oplus v_3 \]
for $v_2,v_3\in B_2$. Then \eqref{eq_beta_fe_eg} reads 
 \[v_1b_1v^{\ast}_1\oplus v_2b_2v^{\ast}_2=v_4b_1v_4^{\ast}\oplus v_3b_2v^{\ast}_3\]
 for every $b=b_1\oplus b_2\in B$.
 The $v_i$ are necessarily unitary, since $W$ is.
 It follows that \eqref{eq_beta_fe_eg} is satisfied exactly when $v_4^{\ast}v_1\in B_1'$ and 
$v^{\ast}_3v_2\in B'_2.$

 We now show that $\alpha$ is not a product of the $*$-automorphism $\beta$ and a $*$-automorphism 
on 
$M_2$.
By direct calculation, for every $m=\begin{bmatrix} m_1 & m_2 \\ m_3 & m_4\\ \end{bmatrix} \in 
M_2$,
\[
 \alpha(1_{B}\otimes m)=\begin{bmatrix}
                                m_1 1_{B_1}              &0 & m_2 v_1v^{\ast}_41_{B_1} & 0\\
                                0                        &m_41_{B_2} & 0 & m_3 
								v_2v^{\ast}_31_{B_2}\\
                                m_3v_4v_1^{\ast}1_{B_1} &0 & m_41_{B_1}&0 \\
                                0                                 & m_2v_3v^{\ast}_21_{B_2} &0 & 
								m_11_{B_2}.\\
                                \end{bmatrix}
\]

 So, $\alpha(1_{B}\otimes m)$ is not of the form 
 \[1_{B}\otimes t=\begin{bmatrix}
                                t_1 1_{B}& t_2 1_{B}\\
                                 t_3 1_{B}& t_4 1_{B}\\
                                \end{bmatrix}.
\]
 Thus, $\alpha$ 
cannot be 
a tensor product of dynamics on $B$ and $M_2$, respectively, unless $B_1=0$ and
$v_2v^*_3=v_3v^*_2=1_{B_1}$, or $B_2=0$ and $v_1v^*_4=v_4v^*_1=1_{B_1}$.
 
 Now consider a specific case.  Let $B_1$ be the group von Neumann algebra generated from a
 free group $G$ on two symbols $c$ and $d$. Let 
$\nu_1$ be the trace on $B_1$ (Example \ref{ex_skew_product}). The map 
$\beta_1:B_1\rightarrow B_1:a\mapsto l(d)a l(d)^{\ast}$ is a $*$-automorphism of $B_1.$ 
Furthermore, since 
$\nu_1$ is a trace, $\nu_1(\beta_1(b_1))=\nu_1(b_1).$ Note that in the cyclic representation 
$(\ell^2(G),\mathrm{id}, \delta_1)$, with $1\in G$ the identity, the unitary representation of $\beta_1$
is given by 
\[U_{\beta_1}\delta_g=U_{\beta_1}l(g)\delta_1=\delta_{dgd^{-1}}\]
for all $g\in G$ (i.e. $U_{\beta_1}=l(d)r(d)$ where $r$ is the right regular representation of
$G$).
Assume that $B_2\neq 0$.
 
 Let $v_1=v_4:=l(d)$. Then we show that $\mathbf{B}$ 
is not compact.
 If we consider the orbit $U_{\beta_1}^\mathbb{Z}\delta_c$ of $\delta_c$ under $U_{\beta_1}$
 \[U_{\beta_1}^\mathbb{Z}\delta_c=\{\ldots, \delta_{d^{-2}cd^{2}}, \delta_{d^{-1}cd^{1}},  
\delta_{c}, 
\delta_{dcd^{-1}}, \delta_{d^{2}cd^{-2}}, \delta_{d^{3}cd^{-3}},\ldots\},\] 
then we have $d^mcd^{-m}\neq d^ncd^{-n},$ and
\begin{align}
 \norm{\delta_{d^mcd^{-m}}-\delta_{d^ncd^{-n}}}=\sqrt{2}\nonumber
\end{align}
for all $m,n\in \mathbb{Z}$ with $m\neq n.$ 
Hence, $U_{\beta_1}^\mathbb{Z}\delta_c$ cannot be totally bounded, so that, as we are in a metric 
space, the closure of 
$U_{\beta_1}^\mathbb{Z}\delta_c$ cannot be compact. 
It follows that $\mathbf{B }$ is not a compact system, i.e. $\mathbf{B }$ does not have discrete spectrum.

Thus we have constructed a finite extension $\mathbf{A }$ of a non-compact system $\mathbf{B }$, 
such that $\alpha$ is not the product of the dynamics on $B $ with the 
dynamics on $M_2$.
\end{example}

It ought to be possible to take an infinite direct sum of copies of $A$ above, to obtain an isometric 
extension of $\mathbf{B}$ which is not a finite extension, by
weighing the traces of the copies of $A$ by weights adding up to one, and allowing for possibly different finite extension dynamics on the copies of $A$. However, the foregoing finite extension 
already makes our main point, namely, it gives a purely noncommutative example of relative discrete spectrum.
\section{Further Questions}\label{sec_questions}
 We end the paper with an informal discussion of some problems related to isometric extensions.
 
 We can consider an intermediate system between a system and an isometric extension of it, and ask 
if the intermediate system leads to two new isometric extensions. (In the classical 
theory such a result holds; see \cite[Lemma 9.12]{G}). In the noncommutative case it can be shown that 
the intermediate system is an isometric
extension of the system, but the question is if  
the original isometric extension is also an isometric extension of the intermediate system. One obstacle is 
relating 
the 
modules of the different pairings with one another. 

A 
technical problem when using our definition of isometric extension is deciding if a given 
projection in the basic construction has finite 
trace. 

Is it possible to make sense of our definition in a more classical sense? For instance, we would like to know 
if there is a sensible notion of generalized eigenvalue. Generalized eigenvectors appear to be ``virtual 
objects'' in our definition and it would be interesting to see if we could find an equivalent 
characterization directly in terms of generalized eigenvectors. 
\section*{Acknowledgements}

This work was partially supported by the National Research Foundation of South
Africa.

\end{document}